\documentclass[11pt]{amsart}      
\usepackage{latexsym,mathtools}
\usepackage{epsfig,setspace}
\usepackage{a4}
\usepackage[utf8]{inputenc}
\usepackage[T1]{fontenc} 
\usepackage{tipa}
\usepackage{amssymb}
\usepackage{amsthm}
\usepackage{amsbsy}
\usepackage{upgreek}   
\usepackage{relsize}
\usepackage{tikz-cd}
\usepackage{ytableau}
\usepackage{amsfonts,amssymb,latexsym,amscd,amsmath,euscript,enumerate,verbatim,calc}
\allowdisplaybreaks[1]
\usepackage{url}
\usepackage{hhline}
\usepackage{pifont}
\usepackage{relsize}
\usepackage[toc,page]{appendix} 
\usepackage[colorlinks=true,linkcolor=blue,citecolor=magenta]{hyperref}      

\theoremstyle{definition}

\newtheorem{theorem}{Theorem}[section]
\newtheorem{lemma}[theorem]{Lemma}
\newtheorem{proposition}[theorem]{Proposition}

\newtheorem{corollary}[theorem]{Corollary}
\newtheorem{definition}[theorem]{Definition}

\newtheorem{remark}[theorem]{Remark}

\def\Fq{{\mathbb F}_q}

\def\I{\mathcal{I}}

\def\CC{\mathcal{C}}
\def\DD{W}
\newcommand{\WW}{W}

\makeatletter
\@namedef{subjclassname@2020}{%
  \textup{2020} Mathematics Subject Classification}
\makeatother

\makeatletter
\def\imod#1{\allowbreak\mkern10mu({\operator@font mod}\,\,#1)}
\makeatother

\title{Enumerating partial linear transformations in a similarity class} 
\author{Akansha Arora}
\address{Indraprastha Institute of Information Technology Delhi (IIIT-Delhi), New Delhi 110020, India.}
\email{akanshaa@iiitd.ac.in}

\author{Samrith Ram}
\address{Indraprastha Institute of Information Technology Delhi (IIIT-Delhi), New Delhi 110020, India.}
\email{samrith@gmail.com}  



\begin{document}
\begin{abstract}
Let $V$ be a finite-dimensional vector space over the finite field ${\mathbb F}_q$ and suppose $W$ and $\widetilde{W}$ are subspaces of $V$. Two linear transformations $T:W\to V$ and $\widetilde{T}:\widetilde{W}\to V$ are said to be similar if there exists a linear isomorphism $S:V\to V$ with $SW=\widetilde{W}$ such that $S\circ T=\widetilde{T}\circ S $. Given a linear map $T$ defined on a subspace $W$ of $V$, we give an explicit formula for the number of linear maps that are similar to $T$. Our results extend a theorem of Philip Hall that settles the case $W=V$ where the above problem is equivalent to counting the number of square matrices over ${\mathbb F}_q$ in a conjugacy class. 
\end{abstract}

\maketitle
\tableofcontents   

\section{Introduction}


Denote by $\Fq$ the finite field with $q$ elements where $q$ is a prime power. Let $\Fq[x]$ denote the ring of polynomials over $\Fq$ in the indeterminate $x$. Throughout this paper $n$ and $k$ denote nonnegative integers. A partition of a nonnegative integer $n$ is a sequence $\lambda=(\lambda_1,\lambda_2,\ldots)$ of nonnegative integers with $\lambda_i\geq \lambda_{i+1}$ for $i\geq 1$ and $\sum_i \lambda_i=n$. If $\lambda_{\ell+1}=0$ for some integer $\ell$, we also write $\lambda=(\lambda_1,\ldots,\lambda_\ell)$. The notation $\lambda \vdash n$ or $|\lambda|=n$ will mean that $\lambda$ is a partition of the integer $n$.

Let $V$ be an $n$ dimensional vector space over $\Fq$ and let $W$ be a subspace of $V$. Let $L(W,V)$ denote the vector space of all $\Fq$-linear transformations from $W$ to $V$. Two linear transformations  $T \in L(W,V)$ and $\widetilde{T} \in L(\widetilde{W},V)$ defined on subspaces $W$ and $\widetilde{W}$ of $V$ respectively are similar if there exists a linear isomorphism $S:V\to V$ 
such that the following diagram commutes: 
\[
\begin{tikzcd}
W \arrow{r}{T} \arrow[swap, twoheadrightarrow]{d}{S} & V \arrow{d}{S}[swap]{\simeq} \\
\widetilde{W} \arrow{r}{\widetilde{T}} & V
\end{tikzcd}.
\]
Let $\mathcal{L}(V)$ denote the union of the vector spaces $L(W,V)$ as $W$ varies over all possible subspaces of $V$. Given $T\in \mathcal{L}(V)$ define $\CC(T)$, the conjugacy class of $T$, by
$$
\CC(T):=\{\widetilde{T}: \widetilde{T}\in \mathcal{L}(V) , \mbox{ } \widetilde{T} \mbox{ is similar to } T\}.
$$

We are interested in determining the cardinality of $\mathcal{C}(T)$ for an arbitrary linear map $T$. The case where $T$ is a linear operator on $V$ is well-studied. Given such a linear operator $T$, one can view $V$ as an  $\Fq [x]$-module where the element $x$ acts on $V$ as the linear transformation $T$. By the structure theorem for modules over a principal ideal domain \cite[p. 86]{MR0369381}, $V$ is isomorphic to a direct sum    
 $$ V \simeq \bigoplus_{i=1}^{r} \frac{\Fq [x]}{(p_i)} $$   
of cyclic modules where $p_1,p_2, \ldots, p_r$ are monic polynomials of degree at least one over $\Fq$ with $p_i$ dividing  $p_{i+1}$ for $1\leq i\leq r-1$.
The $p_i$ are known as the invariant factors of $T$ and uniquely determine $T$ upto similarity;
two linear operators $T$ and $\widetilde{T}$ on $V$ are similar if and only if they have the same invariant factors. In this case the problem of determining $|\CC(T)|$ is equivalent to counting the number of square matrices over $\Fq$ in a conjugacy class. An explicit formula  \cite[Eq. 1.107]{Stanley2012} for the size of $\CC(T)$ for a linear operator $T$ was given by Philip Hall based on earlier work by Frobenius. This problem has also been studied by Kung \cite{Kung1981} and Stong~\cite{Stong1988} who employ a generating function approach. In particular, Kung introduced a vector space cycle index which is an analog of the Pólya cycle index and can be used to enumerate many classes of square matrices over a finite field. We refer to the survey article of Morrison \cite{MR2217227} for more on this topic.   
The invariant factors $p_i$ of a linear operator $T$ appear as the nonunit diagonal entries in the Smith Normal Form \cite[p. 257]{MR0276251} of $xI-A$ where $A$ is the matrix of $T$ with respect to some ordered basis for $V$. 

In this paper we determine the size of the similarity class $\CC(T)$ for an arbitrary transformation $T\in \mathcal{L}(V)$. Our methods are mostly combinatorial and we use ideas from the theory of integer partitions. The first step is to characterize the similarity invariants for a linear transformation $T$ defined only on a subspace $W$ of an $n$-dimensional vector space $V$. Accordingly, let $T\in \mathcal{L}(V)$ be a linear transformation and let $U$ denote the maximal $T$-invariant subspace with $\dim U=d$. Interestingly, in this case the similarity classes are indexed by pairs $(\lambda,\mathcal{I})$ where $\lambda$ is an integer partition of $n-d$ and $\mathcal{I}$ is an ordered set of monic polynomials corresponding to the invariant factors of the restriction of $T$ to $U$. The precise details are in Section~\ref{sec:invariants}. When the domain of $T$ is all of $V$, the partition $\lambda$ above is empty and the similarity class $\mathcal{C}(T)$ is completely determined by the invariant factors of $T$. We prove (Corollary \ref{cor:simsize}) that the size of the conjugacy class corresponding to the pair $(\lambda,\I)$ is given by
$$
 \left| \CC(\lambda,\I )\right| = q^{d(k-d)+\sum_{i\geq 2} \lambda_i^2}  \left|\CC(\I)\right| {n \brack k}_q {k \brack d}_q   \mathlarger\prod_{i\geq 1}{\lambda_i \brack \lambda_{i+1}}_q \prod_{i=0}^{k-d-1} (q^{k-d}-q^i),
$$
where $k=n-\lambda_1$ and ${\cdot \brack \cdot}_q$ denotes a $q$-binomial coefficient while $|\CC(\I)|$ denotes the number of square matrices in the conjugacy class specified by $\I$. In fact Hall's result on matrix conjugacy class size may be recovered from Theorem~\ref{th:givenU} as well as Corollaries \ref{cor:sizeofcwv} and \ref{cor:simsize} by setting $\lambda$ to be the empty partition.  

While the problem of estimating similarity class sizes in $\mathcal{L}(V)$ seems quite natural and is an interesting combinatorial problem in its own right, it also has some connections with mathematical control theory. As a consequence of our results we give another proof of a theorem of Lieb, Jordan and Helmke \cite[Thm.~1]{Helmkeetal2015} which is related to the problem of counting the number of zero kernel pairs of matrices or, equivalently, reachable linear systems over a finite field. This problem was initially considered by Koci\textpolhook{e}cki and Przyłuski~\cite{MR1019984}. The reader is referred to \cite{zerokernel,RAM2017146} for the definition of zero kernel pairs and the connections with control theory.      

\section{Similarity invariants for maps defined on a subspace} 
\label{sec:invariants}

We begin by describing a complete set of similarity invariants for a linear map in $L(W,V)$. Given $T\in L(W,V)$, define a sequence of subspaces \cite[sec. III.1]{Partially} $\DD_i=\DD_i(T)(i\geq 0)$ by $\DD_0=V, \DD_1=W$ and
\begin{align*} 
  \DD_{i+1}=\DD_i\cap T^{-1}(\DD_i)=\{v\in \DD_i:Tv\in \DD_i\} \quad \mbox{ for }i\geq 1.
\end{align*}
 The descending chain of subspaces $\DD_0\supseteq \DD_1\supseteq \cdots$ eventually stabilizes as the dimensions of the subspaces are nonnegative integers. Let $d_i=d_i(T):=\dim \DD_i$ for $i\geq 0$ and let
$$
\ell=\ell(T):=\min \{i:\DD_i=\DD_{i+1}\}.
$$
The subspace $\DD_{\ell}$ is clearly a $T$-invariant subspace which is evidently the maximal $T$-invariant subspace. 
Therefore the restriction $T_{W_\ell}$ of $T$ to $\DD_\ell$ is a linear operator on $W_\ell$. Denote by $\I_T$ the ordered set of invariant factors of $T_{W_\ell}$. 
Since the characteristic polynomial of $T_{W_\ell}$ equals the product of the invariant factors of $T_{W_\ell}$, it follows that
$$
d_{\ell}=\deg\prod_{p\in \I_T}p.
$$
Now define
$$
\lambda_j=\lambda_j(T):=d_{j-1}-d_{j} \mbox{ for } 1\leq j \leq \ell.
$$
\begin{definition}
The integers $\lambda_j(T)(1\leq j\leq \ell)$ are called the \emph{defect dimensions} \cite[p. 52]{Partially} of $T$. 
\end{definition}

\begin{lemma}
  For any $T\in \mathcal{L}(V)$, we have $\lambda_j(T)\geq \lambda_{j+1}(T)$ for $ 1\leq j\leq \ell-1$.
\end{lemma}
\begin{proof}
Let the subspaces $W_j(j\geq 1)$ be as above. Note that $T(W_j)\subseteq W_{j-1}$ for each $j$. Fix $j\geq 1$ and define a map $\varphi: \DD_j/ \DD_{j+1} \to \DD_{j-1}/ \DD_j$ by
  $$
\varphi(v+  \DD_{j+1})=Tv+ \DD_j.
$$
We claim that $\varphi$ is well defined. Suppose $v_1+\DD_{j+1}=v_2+ \DD_{j+1}$ for some $v_1,v_2\in W_j$. Then $v_1-v_2\in \DD_{j+1}$ and consequently $T(v_1-v_2)\in \DD_{j}$. Therefore $Tv_1+\DD_j=Tv_2+\DD_j$ proving that $\varphi$ is well defined. The linearity of $\varphi$ follows easily from the fact that $T$ is linear. In fact $\varphi$ is also injective. Suppose for some $v\in \DD_j$ we have $$\varphi(v+\DD_{j+1})=Tv+ \DD_j=0+\DD_j.$$
Then $Tv\in \DD_j$ and since $v$ itself lies in $\DD_j$, it follows that $v\in \DD_{j+1}$ as well. Thus $v+ \DD_{j+1}$ is the zero vector and $\varphi$ is injective. The injectivity of $\varphi$ implies that $ \dim (\DD_{j-1}/\DD_{j})\geq \dim (\DD_j/\DD_{j+1})$, or equivalently, $\lambda_j\geq \lambda_{j+1}$ for $1\leq j\leq \ell -1$.
\end{proof}
Hereon the sequence $W_i(T)(i\geq 0)$ will be referred to as the \emph{chain of subspaces} associated with $T$.
\begin{corollary}
For $T\in \mathcal{L}(V)$, let $\ell=\ell(T)$. The sequence $\lambda_T=(\lambda_1(T),\ldots,\lambda_{\ell}(T))$ is an integer partition of $n-d_{\ell}(T)$.
\end{corollary}
\begin{proof}
  This follows since $\sum_{i=1}^{\ell}\lambda_i=d_0-d_\ell=n-d_\ell$.
\end{proof}
We will prove that the pair $(\lambda_T,\I_T)$ completely determines the similarity class of a linear transformation $T$ in the sense that two maps $T,\widetilde{T}\in \mathcal{L}(V)$
 are similar if and only if $\lambda_{T}=\lambda_{\widetilde{T}}$ and $\I_{T}=\I_{\widetilde{T}}$. We require a lemma \cite[Ch. III Lem. 3.3]{Partially} to prove this result. As the terminology in \cite{Partially} differs considerably from that in this paper, we include a proof here for the sake of completeness.

 \begin{lemma}
   \label{lem:restriction}
  Let $W,\widetilde{W}$ be subspaces of $V$. For $T \in L(W,V)$ and  $ \widetilde {T}  \in L(\widetilde{W} ,V)$, let $T_{U}$ and $\widetilde{T}_{\widetilde{U}}$ denote the restrictions of $T$ and $\widetilde{T}$ to the subspaces
  $$U =\{ v \in W: Tv \in W \} \mbox{ and } \widetilde{U}=\{ v \in \widetilde{W}:\widetilde{T}v \in \widetilde{W} \}$$
  respectively. Then $T$ is similar to $\widetilde{T}$ if and only if $T_U$ is similar to $\widetilde{T}_{\widetilde{U}}$ and $\dim W= \dim \widetilde{W}$.
\end{lemma}
\begin{proof} 
  First suppose that $T$ is similar to $\widetilde{T}$. Then there exists a linear isomorphism $S:V \rightarrow V$ such that $SW = \widetilde{W}$ and $S \circ T = \widetilde{T} \circ S$. It follows that dim $W=$ dim $\widetilde{W}$. We claim that $T_U$ is similar to $\widetilde{T}_{\widetilde{U}}$ with respect to the same linear isomorphism $S$. To see this, we first show that  $S$ maps $U$ onto $\widetilde{U}$. Suppose $v \in U$. Then, by definition, $ v \in  W $ and $Tv \in  W$. This implies that $Sv \in \widetilde{W} $ and consequently $\widetilde{T}\circ Sv= S\circ Tv \in \widetilde{W}$ which further implies that $Sv \in \widetilde{U}$. Thus $S U \subseteq \widetilde{U}$. 
  Now the isomorphism $S^{-1}:V\to V$ has the property that $S^{-1}\widetilde{W}=W$ and $S^{-1}\circ \widetilde{T}=T\circ S^{-1}$. By reasoning as above it follows that $S^{-1}\widetilde{U}\subseteq U$. It follows that $SU=\widetilde{U}$. Now since $T_{U}$ and $\widetilde{T}_{\widetilde{U}}$ are restrictions of $T$ and $\widetilde{T}$ to $U$ and $\widetilde{U}$ respectively, it is easy to see that $S\circ T_U = \widetilde{T}_{\widetilde{U}}\circ S$ and it follows that $T_U$ and $\widetilde{T}_{\widetilde{U}}$ are similar.
  
  For the converse, suppose dim $W=$ dim $\widetilde{W}$ and $T_U$ is similar to $\widetilde{T}_{\widetilde{U}}$. This implies that there exists a linear isomorphism $S' \in GL(V) $ such that $S'\circ U= \widetilde{U}$ and $S'\circ T_U = \widetilde{T}_{\widetilde{U}}\circ S'$. First construct a linear isomorphism $S '' \in GL(V)$ such that $S ''W = \widetilde{W}$ and $S ''\circ T_U = \widetilde{T}_{\widetilde{U}}\circ S''$. Note that $T(U) \subseteq W$. We simply set $S''v =S'v$  for all $v$ lying in the subspace $U + TU$ of $W$. Since $S'(u_1 + Tu_2)=S'u_1 + S'\circ T_U u_2=S'u_1 + \widetilde{T}_{\widetilde{U}}\circ S' u_2 \in \widetilde{U} + \widetilde{T}\widetilde{U}$, it is clear that $S'': U + TU \rightarrow \widetilde{U} + \widetilde{T}\widetilde{U}$ is an isomorphism. Since dim $W=$ dim $\widetilde{W}$, we may extend the definition of $S''$ to all of $W$ to obtain a linear isomorphism $S'':W \rightarrow \widetilde{W}$ which may be further extended to a linear isomorphism $S'':V \rightarrow V$.

  Now we use $S''$ to construct another linear isomorphism $S:V\to V$ such that $SW = \widetilde{W}$ and $S\circ T = \widetilde{T}\circ S$ which will imply that the linear transformtions $T$ and $\widetilde{T}$ are similar. Let $Sv= S'' v$  for any $v \in W$ and let $Sv'= \widetilde{T}\circ S'' v$  for any $v'=Tv \in TW$. We assert that $S: W + TW \rightarrow \widetilde{W} + \widetilde{T}\widetilde{W}$ is well defined and a linear isomorphism. If $v'=Tv$ lies in $ W$, then $v \in U$ and hence $  S''v' = S''\circ Tv = S''\circ T_U v= \widetilde{T}_{\widetilde{U}}\circ S''v = \widetilde{T}\circ S''v $. Therefore $Sv'$ is uniquely defined.  If $Tv = Tu$ for some $v,u \in W$, then $T(v-u)=0$ lies in $W$. Thus, $ S\circ T(v-u)=S''\circ T(v-u)=0$ which further implies $ \widetilde{T}\circ S'v - \widetilde{T}\circ S'u = \widetilde{T}\circ S'(v-u)=0$, and hence $S\circ Tv=S\circ Tu$. This implies that $S$ is well defined. To prove that $S$ is injective, let $v'= Tv$ for some $v \in W$ and $ Sv'=0$. This implies $Sv'=\widetilde{T}\circ S''v=0$ which further implies $ S''\circ Tv=0$ and since $S''$ is invertible, it follows $v'=0$. It is easy to check that $S$ is surjective and $S\circ T = \widetilde{T}\circ S$. Furthermore, it can be extended to a linear isomorphism $S:V \rightarrow V$. This completes the proof. 
\end{proof} 
\begin{proposition}
\label{prop:similarity}
 The linear transformations $T \in L(W,V)$ and $\widetilde{T}  \in L(\widetilde{W},V)$ are similar if and only if $\lambda_{T}= \lambda_{\widetilde{T} }$ and $\I_{T}= \I_{\widetilde{T} }$.
\end{proposition}
\begin{proof} For $T\in L(W,V)$, consider the sequence of subspaces $W_i$ such that $W_0=V$, $W_1=W$ and $W_{i+1}= \{ v \in W_i: Tv \in W_i \}.$    Let $\ell= \text{min}\{i: W_i=W_{i+1}\}$ and denote by $T_i$ the restriction of $T$ to $W_i$ for $1\leq i \leq \ell$. Similarly, define $\widetilde{W}_i$, $\widetilde{T}_i$, $\widetilde{\ell}$ for $\widetilde{T}  \in L(\widetilde{W} ,V)$. By Lemma \ref{lem:restriction}, it follows that $T_1$ is similar to $\widetilde{T}_1$ if and only if $T_{2}$ is similar to $\widetilde{T}_{2}$ and dim $W_1=$ dim $\widetilde{W}_1$. 
Using the lemma again, it is clear that $T_1$ is similar to $\widetilde{T}_1$ if and only if $T_{3}$ is similar to $\widetilde{T}_{3}$, dim $W_{2}=$ dim $\widetilde{W}_{2}$  and  dim $W_1=$ dim $\widetilde{W}_1$. By repeated application of the lemma, it is evident that $T $ is similar to $\widetilde{T} $ if and only if
$T_\ell$ is similar to $\widetilde{T}_{\widetilde{\ell}}$ with  $\ell= \widetilde{\ell}$ and dim $W_i=$ dim $\widetilde{W}_i$ for $1 \leq i \leq \ell$. The linear operators $T_\ell: W_{\ell} \rightarrow W_{\ell} $ and  $\widetilde{T}_{\widetilde{\ell}}: \widetilde{W}_{\widetilde{\ell}} \rightarrow \widetilde {W}_{\widetilde{\ell}} $ are similar if and only if $\I_{T}= \I_{\widetilde{T}}$. Thus, it follows that $ T $ and $\widetilde{T} $ are similar if and only if $\lambda_{T}= \lambda_{\widetilde{T}}$ and $\I_{T}= \I_{\widetilde{T}}$. 
\end{proof}
\begin{definition}
For any ordered set of invariant factors $\I$, define
$$\deg \I=\deg\prod_{p\in \I}p.$$
\end{definition}

\begin{remark}
  In view of the above proposition similarity classes in $\mathcal{L}(V)$ are indexed by pairs $(\lambda,\I)$ where $\lambda$ is an integer partition (possibly the empty partition) and $\I\subseteq \Fq[x]$ is an ordered set of invariant factors satisfying
$$
|\lambda|+\deg \I=\dim V.
$$
Denote the similarity class in $\mathcal{L}(V)$ corresponding to the pair $(\lambda,\I)$ by $\mathcal{C}(\lambda,\I)$. 
For a given subspace $W$ of $V$  and an integer partition $\lambda$ with largest part $\dim V - \dim W$, denote by  $\mathcal{C}_{W,V} (\lambda,\I) $ the set of all linear transformations in $L(W,V)$ corresponding to the pair  $(\lambda,\I)$, i.e.,
$$\mathcal{C}_{W,V} (\lambda,\I)  := L(W,V) \cap \mathcal{C}(\lambda,\I). $$
In the case $W=V$, the similarity class $\mathcal{C}_{V,V} (\lambda,\I) $ is defined only when $\lambda$ is the empty partition and it depends only on the invariant factors $\I$. In this case $\CC_{V,V}(\emptyset,\I)$ is abbreviated to $\CC(\I)$. A closed formula for the size of $\CC(\I)$  can be found in Stanley \cite[Eq. 1.107]{Stanley2012}.
\end{remark}

\section{Counting simple linear transformations} 
\begin{definition}          
\label{simpletransformation}
A linear transformation $T\in \mathcal{L}(V)$ is \emph{simple} if, for each $T$-invariant subspace $U$, either $U=\{0\}$ or $U=V$. 
\end{definition}
It follows from the definition that simple maps are injective. If $T\in \mathcal{L}(V)$ is simple with domain a proper subspace of $V$, then the maximal $T$-invariant subspace is necessarily the zero subspace and therefore $T\in \mathcal{C}(\lambda,\emptyset)$ for some integer partition $\lambda$ of $\dim V$ with largest part $\dim V-\dim W$ where $W$ is the domain of $T$. In this section we determine the size of $\mathcal{C}(\lambda,\emptyset)$ for an arbitrary partition $\lambda$ of $\dim V$. We begin with some combinatorial lemmas.
 

The number of $k$-dimensional subspaces of an $n$-dimensional vector space over $\Fq$ is given by the $q$-binomial coefficient \cite[p. 292]{MR1207813}    
$$
{n \brack k}_q:=\mathlarger\prod_{i=1}^k \frac{q^{n-i+1}-1}{q^i-1}. 
$$

\begin{lemma}
\label{lem:nointersectionsubspace }
Let $U\subseteq W$ be subspaces of an $n$-dimensional vector space $V$ over $\Fq$ with $\dim U=d$ and $\dim W=k$. The number of $k$-dimensional subspaces of $V$ whose intersection with $W$ is $U$ equals
  $$
 {n-k \brack k-d}_q q^{(k-d)^2}.
  $$
\end{lemma}
\begin{proof}
  We count the number of $k$-dimensional subspaces $W'$ for which $W \cap W'=U$. Given any ordered basis of $U$, there are $\prod_{i=k}^{2k-d-1}(q^n-q^i)$ ways to extend it to an ordered basis of $W'$. Counting in this manner, the same subspace $W'$ arises in precisely $\prod_{i=d}^{i=k-1}(q^k-q^i)$ ways. Thus the total number of such subspaces $W'$ is given by  
  $$
  \frac{\prod_{i=k}^{2k-d-1}(q^n-q^i)}{\prod_{i=d}^{i=k-1}(q^k-q^i)}= {n-k \brack k-d}_q q^{(k-d)^2}. 
  $$
\end{proof}
\begin{definition}
A \emph{flag} \cite[p. 95]{MR1153249} of length $r$ in a vector space $V$ is an increasing sequence of subspaces $W_i(0\leq i\leq r)$ such that
$$   \{0\} = \WW_{0} \subset \WW_{1} \subset\cdots \WW_{r-1} \subset \WW_{r} =V. $$
\end{definition}
The following lemma \cite[Sec. 1.5]{MR2217227} gives the number of flags of length $r$ with subspaces of given dimensions.
 
\begin{lemma}
 \label{lem:noflags}
Let $n_1,\ldots,n_r$ be positive intgers with  $n_1+\cdots+n_r=n$. The number of flags $W_0\subset \cdots \subset W_r$ of length $r$ in an $n$-dimensional vector space $V$ over $\Fq$ with $\dim W_i= n_1 +n_2+ \cdots + n_i$ is given by the $q$-multinomial coefficient

$${n \brack n_1, n_2, \ldots, n_r}_q: = \frac{[n]_q!}{[n_1]_q![n_2]_q!\ldots[n_r]_q!},$$
where $[n]_q :=\frac{q^n-1}{q-1} $ and $[n]_q!:= [n]_q [n-1]_q \ldots [1]_q.$
 
\end{lemma}

In the statement of the following theorem and the rest of this paper, the number of nonsingular $k\times k$ matrices over $\Fq$  \cite[Sec. 1.2]{MR2217227} is denoted by $\gamma_q(k)=\prod_{i=0}^{k-1} (q^k-q^i)$.
\begin{theorem}
 Let $\lambda$ be a partition of $n$. Then
  $$
 \left|\CC(\lambda,\emptyset)\right|= q^{\sum_{j\geq 2}\lambda_j^2}{n \brack n-\lambda_1, \lambda_1-\lambda_2,\ldots,\lambda_{\ell}}_q \gamma_q(n-\lambda_1). 
$$
 
\end{theorem}
\begin{proof}
  We count the number of simple linear transformations $T\in \mathcal{L}(V)$ having defect dimensions $(\lambda_1,\lambda_2, \ldots, \lambda_{\ell})$ defined on some subspace of $V$ of dimension $n-\lambda_1$. Fix a subspace $W$ of $V$ of dimension $n-\lambda_1$. We first determine the cardinality of $\CC_{W,V} (\lambda,\emptyset)$. For $T \in \CC_{W,V}(\lambda,\emptyset)$ consider the chain of subspaces $\{W_i=W_i(T) \}_{i=0}^{\ell}$ associated with $T$. Define a sequence $\{ \WW_i' \}_{i=1}^{\ell}$ by $ \WW_i'=T(\WW_i)$. Since $T$ is injective, we have $\dim \WW_i= \dim \WW_i' =d_i$ for $i\geq 1$. By the choice of $T$ we have $d_{i-1}-d_i=\lambda_i$. Note that $\WW_{i} \cap {\WW}_{i}'= \WW_{i+1}'$ for $1 \leq i \leq \ell-1$. 
  The sequence $\{\WW_i\}_{i=0}^\ell$ is a flag 
in $W$ of length $\ell-1$:
$$   \{0\} = \WW_{\ell} \subset \cdots \WW_2 \subset \WW_1 =W $$
 such that dim $\WW_i =d_i= \lambda_{\ell} + \lambda_{\ell-1} + \cdots \lambda_{i+1}   $. By Lemma \ref{lem:noflags}, the number of such flags is $${n-\lambda_1 \brack \lambda_{\ell}, \lambda_{\ell-1}, \ldots , \lambda_2 }_q .$$ 
For a given choice of $\{\WW_i\}_{i=0}^{\ell}$, the total number of choices for the sequence $\{\WW_i'\}_{i=0}^{\ell}$ equals the total number of flags
 $$   \{0\} = \WW_{\ell}' \subset \cdots \WW_2' \subset \WW_1' = TW $$
  of length $\ell-1$ where dim $\WW_i' =d_i$ and $\WW_{i} \cap \WW_{i}'= \WW_{i+1}'$ for $1 \leq i \leq \ell-1$. 
Thus $W_{\ell-1}'$ is a subspace of $W_{\ell-2}$ of dimension $d_{\ell-1}$ that intersects $W_{\ell-1}$ trivially. It follows by Lemma \ref{lem:nointersectionsubspace } that $\WW_{\ell-1}'$ can be chosen in   
$$ { d_{\ell-2} -d_{\ell-1} \brack d_{\ell-1} - d_{\ell}}_q q^{(d_{\ell-1} - d_{\ell})^2} ={ \lambda_{\ell-1} \brack \lambda_{\ell}}_q q^{\lambda_{\ell}^2}$$
ways. Similarly, the conditions $\WW_{\ell-2}' \subseteq \WW_{\ell-3}$ and 
$\WW_{\ell-2} \cap \WW_{\ell-2}'= \WW_{\ell-1}'$ imply that $\WW_{\ell-2}'$ can be chosen in
$${ d_{\ell-3} -d_{\ell-2} \brack d_{\ell-2} - d_{\ell-1}}_q q^{(d_{\ell-2} - d_{\ell-1})^2} = { \lambda_{\ell-2} \brack \lambda_{\ell-1}}_q q^{\lambda_{\ell-1}^2}$$
ways. Proceeding in this manner, it is seen that the total number of choices for the sequence $\{\WW_i'\}_{i=1}^\ell$ is equal to $${ \lambda_{\ell-1} \brack \lambda_{\ell}}_q q^{\lambda_{\ell}^2} \, { \lambda_{\ell-2} \brack \lambda_{\ell-1}}_q q^{\lambda_{\ell-1}^2}\,  \cdots \, { \lambda_{1} \brack \lambda_{2}}_q q^{\lambda_{2}^2}= q^{\sum_{i=2}^{\ell} \lambda_i^2} { \lambda_{1} \brack \lambda_{1}-\lambda_{2}, \lambda_{2}-\lambda_{3}, \ldots , \lambda_{\ell}  }_q.$$
For each choice of the flags $\{\WW_i\}_{i=0}^\ell$ and $\{\WW_i'\}_{i=0}^\ell$, we count the number of possibilities for $T$. Note  that $T$ is injective and $T \WW_i= \WW_i'$ for $1 \leq i \leq \ell$. Thus the number of  ways to map $\WW_{\ell-1}$ onto $\WW_{\ell-1}'$ is equal to the number of invertible $\lambda_{\ell} \times \lambda_{\ell}$ matrices over $\Fq$, i.e., $\gamma_q({\lambda_{\ell}})$. The number of ways to extend $T$ to $\WW_{\ell-2}$ such that $T \WW_{\ell-2} = \WW_{\ell-2}'$ is evidently

$$\prod_{i=d_{\ell-1}}^{d_{\ell-1}+\lambda_{\ell-1}-1} (q^{d_{\ell-2}}-q^i) = q^{d_{\ell-1}\lambda_{\ell-1}}\gamma_q({\lambda_{\ell-1}}).$$

Following this line of reasoning, the total number of choices for the map $T$ for a given choice of $\{\WW_i\}_{i=0}^\ell$ and $\{\WW_i'\}_{i=0}^\ell$ equals
$$q^{\sum_{i=2}^{\ell} {d_i\lambda_i}}\prod_{i=2}^{\ell}\gamma_q({\lambda_i}).$$
It follows that
\begin{align*}
\left|\mathcal{C}_{W,V} (\lambda,\emptyset)\right| & =  {n-\lambda_1 \brack \lambda_\ell, \lambda_{\ell-1}, \ldots , \lambda_2 }_q  q^{\sum_{i=2}^{\ell}\lambda_i^2} { \lambda_{1} \brack \lambda_{1}-\lambda_{2}, \lambda_{2}-\lambda_{3}, \ldots , \lambda_{\ell}  }_q  q^{\sum_{i=2}^{\ell}  {d_i\lambda_i}}  \\
& \qquad \times \prod_{i=2}^{\ell}\gamma_q({\lambda_i}).
\end{align*}
We expand the values of $\gamma_q({\lambda_i})$  and simplify the above expression.
\begin{align}
\left|\mathcal{C}_{W,V} (\lambda,\emptyset)\right| \nonumber
& =  q^{\sum_{i=2}^{\ell}\lambda_i^2} {\lambda_{1} \brack \lambda_{1}-\lambda_{2}, \lambda_{2}-\lambda_{3}, \ldots , \lambda_{\ell}  }_q \frac{[n-\lambda_1]_q!}{[\lambda_{\ell}]_q! [\lambda_{\ell-1}]_q! \cdots [\lambda_2]_q! } q^{\sum_{i=2}^{\ell}  {d_i\lambda_i}}  \\ \nonumber
& \qquad \times \prod_{i=2}^{\ell} (q-1)^{\lambda_i} q^{\binom{\lambda_i}{2}} {[\lambda_i]_q!}\\ \nonumber
& =  q^{\sum_{i=2}^{\ell}\lambda_i^2} { \lambda_{1} \brack \lambda_{1}-\lambda_{2}, \lambda_{2}-\lambda_{3}, \ldots , \lambda_{\ell}  }_q [n-\lambda_1]_q!q^{\sum_{i=2}^{\ell}  {d_i\lambda_i}} (q-1)^{\lambda_2+\cdots +\lambda_{\ell}}  \\ \nonumber
& \qquad \times  q^{\binom{\lambda_2}{2} + \cdots + \binom{\lambda_{\ell}}{2}} \\ \nonumber
& =  q^{\sum_{i=2}^{\ell}\lambda_i^2} { \lambda_{1} \brack \lambda_{1}-\lambda_{2}, \lambda_{2}-\lambda_{3}, \ldots , \lambda_{\ell}  }_q [n-\lambda_1]_q! (q-1)^{n-\lambda_1} q^{\binom{n-\lambda_1}{2}} \label{eq:CWV}\\ 
 & =  q^{\sum_{i=2}^{\ell}\lambda_i^2} { \lambda_{1} \brack \lambda_{1}-\lambda_{2}, \lambda_{2}-\lambda_{3}, \ldots , \lambda_{\ell}  }_q \gamma_q(n-\lambda_1). 
\end{align}
Since the domain of $T$ is an arbitrary $n-\lambda_1$ dimensional subspace of $V$, we sum over all $(n-\lambda_1)$ dimensional subspaces of $V$ to obtain
\begin{align*}
 \left| \CC(\lambda,\emptyset )\right| &=\sum_{W:\dim {W} =n-\lambda_1} \left| \CC_{W,V}(\lambda,\emptyset )\right|
                         = {n \brack n-\lambda_1}_q  \left|\CC_{W,V}(\lambda,\emptyset )\right|.
 \end{align*} 
 Substituting the expression for $\left|\CC_{W,V}(\lambda,\emptyset )\right|$ obtained earlier, we obtain
 \begin{align*}                        
 \left| \CC(\lambda,\emptyset ) \right|                 &= {n \brack n-\lambda_1}_q q^{\sum_{i=2}^{\ell}\lambda_i^2} { \lambda_{1} \brack \lambda_{1}-\lambda_{2}, \lambda_{2}-\lambda_{3}, \ldots , \lambda_{\ell}  }_q \gamma_q(n-\lambda_1) \\
                         &= q^{\sum_{i=2}^{\ell}\lambda_i^2}{n \brack n-\lambda_1, \lambda_1-\lambda_2,\ldots,\lambda_{\ell}}_q \gamma_q(n-\lambda_1). \qedhere 
\end{align*}                         

\end{proof}

\begin{corollary}
  \label{cor:cwv}
Let $W$ be a proper subspace of an $n$-dimensional vector space $V$ over $\Fq$. Let $\lambda \vdash n$ with $\lambda_1=\dim V-\dim W$. Then the number of simple linear transformations defined on $W$ with defect dimensions $\lambda$ is given by
  $$\sigma(\lambda):=\left| \CC_{W,V} (\lambda,\emptyset)\right| = q^{\sum_{i\geq 2}\lambda_i^2} \gamma_q(n-\lambda_1)  \mathlarger\prod_{i\geq 1}{\lambda_i \brack \lambda_{i+1}}_q.
$$

\end{corollary}
\begin{proof} 
Follows from Equation \eqref{eq:CWV} in the proof of the above theorem. 
\end{proof}
The above corollary may be used to deduce the number of simple linear transformations with a fixed domain by summing $\sigma(\lambda)$ over partitions with a fixed first part. We first collate some basic results on partitions. A useful graphic representation of an integer partition is the corresponding Young diagram. Given a partition $\lambda=(\lambda_1,\lambda_2,\ldots)$, put $\lambda_i$ (unit) cells in row $i$ to obtain its Young diagram. For instance, the Young diagram of the partition $(6,3,2)$ is shown in Figure \ref{fig:ferrers}. 
\begin{figure}
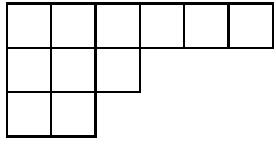

  \centering
  \ydiagram{6,3,2}
 \caption{The Young diagram of $(6,3,2)$.} 
  \label{fig:ferrers}

\end{figure}
\begin{definition}
  For integers $m,r,s$ denote by $p(m,r,s)$ the number of partitions of $m$ with at most $r$ parts in which each part is at most $s$.
\end{definition}
The geometric interpretation of $p(m,r,s)$ is that it counts the number of partitions of $m$ whose Young diagrams fit in a rectangle of size $r\times s$. The following lemma \cite[Prop. 1.1]{MR2339282} shows that the generating function for $p(m,r,s)$ for fixed values of $r$ and $s$ is a $q$-binomial coefficient.  
\begin{lemma}
  \label{lem:box}
  We have
  $$
{r+s \brack s}_q= \sum_{i\geq 0}p(i,r,s)q^i.
  $$
\end{lemma}
The \emph{rank} of a partition $\lambda$ is the largest integer $i$ for which $\lambda_i\geq i$. Geometrically the rank of a partition corresponds to side length of the largest square, called the \emph{Durfee square}, contained in the Young diagram of $\lambda$. The Durfee square of the partition $\lambda=(6,4,3,2)$ is indicated by the shaded cells in Figure~\ref{fig:durfee}.  
\begin{figure}
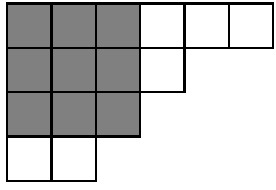

  \centering
  \ytableausetup{nosmalltableaux}
  \begin{ytableau}
    *(gray) & *(gray)  &*(gray) &  & & \\
    *(gray) & *(gray)  &*(gray) &   \\
    *(gray) & *(gray)  &*(gray)  \\
     &  
  \end{ytableau}

  \caption{The Durfee square of the partition $(6,4,3,2)$.}
  \label{fig:durfee}
\end{figure}
\begin{proposition}
  \label{prop:durfees}
  For positive integers $m\leq n$, we have
  $$
 \mathlarger\sum_{\substack{\lambda \vdash n \\ \lambda_1=m}} q^{\sum \lambda_i^2}\mathlarger\prod_{i\geq 1}{\lambda_i \brack \lambda_{i+1}}_q = q^{m^2+n-m} {n-1 \brack m-1}_q. 
  $$
\end{proposition}
\begin{proof}
  Let $S$ denote the set of all partitions $\mu$ of rank $m$ and largest part $n$ with precisely $m$ parts. Visually $S$ consists of partitions whose Young diagrams fit inside an $m\times n$ rectangle $R$ and have at least $m$ cells in each row with precisely $n$ cells in the first row. We compute the sum
  $$
\sum_{\mu \in S}q^{|\mu|}
$$
in two different ways. Note that each $\mu \in S$ is uniquely determined by the partition $\mu'=(\mu_2-m,\mu_3-m,\ldots)$ since the first row and first $m$ columns of the Young diagram of $\mu$ are fixed. As the diagram of $\mu'$ fits in the $m-1\times n-m$ rectangle at the bottom right corner of $R$, it follows by Lemma~\ref{lem:box} that
\begin{align*}
  \sum_{\mu \in S}q^{|\mu|}&=q^{m^2+n-m}\sum_{\mu\in S}q^{|\mu'|}\\
  &=q^{m^2+n-m} {n-1 \brack m-1}_q,
\end{align*}
which accounts for the expression on the right hand side of the proposition. Now for any $\mu\in S $ consider the partition $\varphi(\mu)=\lambda \vdash n$ defined as follows: $\lambda_1$ is the rank of $\mu$, $\lambda_2$ is the rank of the partition whose diagram is to the right of the Durfee square of $\mu$ etc. For example, when $\mu=(8,7,6,5)$, we have $\varphi(\mu)=(4,2,1,1)$ as shown in Figure \ref{fig:durfees}. As $\mu$ varies over $S$, the partition $\varphi(\mu)$ varies over all partitions of $n$ with largest part $m$. Therefore
$$
\sum_{\mu \in S}q^{|\mu|}=\sum_{\substack{\lambda \vdash n\\\lambda_1=m}} \sum_{\substack{\mu\in S\\ \varphi(\mu)=\lambda}} q^{|\mu|}.
$$

\begin{figure}
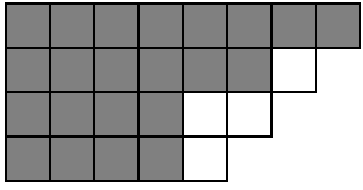

  
  \centering
  \ytableausetup{nosmalltableaux}
  \begin{ytableau}
    *(gray) & *(gray)  &*(gray) &*(gray)  &   *(gray) & *(gray)  &*(gray) &*(gray) \\
    *(gray) & *(gray)  &*(gray) &*(gray) &*(gray) &*(gray) &  \\
    *(gray) & *(gray)  &*(gray) &*(gray) &  & \\
    *(gray) & *(gray)  &*(gray) &*(gray) &  
  \end{ytableau}

  \caption{The partition $\varphi(\mu)=(4,2,1,1)$ corresponding to $\mu=(8,7,6,5)$.}
  \label{fig:durfees}
\end{figure}
Consider the inner sum on the right hand side. If $\varphi(\mu)=\lambda$, then $\lambda$ defines a sequence of squares (corresponding to the shaded cells in Figure \ref{fig:durfees}) which accounts for $\sum_i\lambda_i^2$ cells in the diagram of $\mu$. The cells of $\mu$ that do not lie in any square in the sequence (the unshaded cells in the running example of Figure~\ref{fig:durfees}) correspond to a sequence of partitions: the first is a partition that fits in a rectangle of size $(\lambda_1-\lambda_2)\times \lambda_2$, the second is a partition that fits in a rectangle of size $(\lambda_2-\lambda_3)\times \lambda_3$ etc. Putting these observations together and applying Lemma~\ref{lem:box}, it is clear that
\begin{align*}
  \sum_{\substack{\mu\in S\\ \varphi(\mu)=\lambda}} q^{|\mu|}= q^{\sum \lambda_i^2}{\lambda_1 \brack \lambda_2}_q {\lambda_2 \brack \lambda_3}_q \cdots 
\end{align*}
and the proposition follows. 
\end{proof}
We now deduce the theorem of Lieb, Jordan and Helmke \cite[Thm. 1]{Helmkeetal2015} alluded to in the introduction.

\begin{corollary}
  \label{cor:numsimp}
  Let $W$ be a proper $k$-dimensional subspace of a vector space $V$ of dimension $n$ over $\Fq$. The number of simple linear transformations with domain $W$ equals $
  \prod_{i=1}^k (q^n-q^i).
$
\end{corollary}
\begin{proof}
  The number of simple linear transformations with domain $W$ is equal to
  $$
\sum_{\substack{\lambda \vdash n\\ \lambda_1=n-k}}  \sigma(\lambda)=  \gamma_q(k) \sum_{\substack{\lambda \vdash n\\ \lambda_1=n-k}}  q^{\sum_{i\geq 2} \lambda_i^2} \prod_{i\geq 1}{\lambda_i \brack \lambda_{i+1}}_q 
$$
by Corollary \ref{cor:cwv}. Setting $m=n-k$ in Proposition \ref{prop:durfees} the sum on the right hand side above becomes
\begin{align*}
  q^k {n-1 \brack k}_q   \gamma_q(k)&=q^k  \frac{(q^{n-1}-1)\cdots (q^{n-1}-q^{k-1})}{(q^{k}-1)\cdots (q^{k}-q^{k-1})} \prod_{i=0}^{k-1}(q^k-q^i)\\
  &=\prod_{i=1}^k (q^n-q^i).\qedhere
\end{align*}
\end{proof}
The corollary above can also be obtained \cite[Cor. 2.12]{arora2019unimodular} by counting certain unimodular matrices over a finite field.               
\section{Arbitrary linear transformations defined on a subspace}  
In this section we extend the results obtained on the conjugacy class size of simple linear transformations to arbitrary maps in $\mathcal{L}(V)$. Let $T \in \mathcal{L}(V)$ be a fixed but arbitrary linear transformation with domain $W$ and let $U$ denote the maximal invariant subspace of $T$. Define a map $\hat{T}$ from the quotient space $ W/U$ into $V/U$ by  
 $$\hat{T} ( v+ U) = Tv + U. $$
 
Then $\hat{T}$ is well defined. If $v_1+ U =v_2 + U$ for some $v_1, v_2 \in W$ then $v_1 -v_2 \in U$ and consequently $T(v_1 -v_2) \in U$ since $U$ is $T$-invariant. It follows that $Tv_1 + U= Tv_2 + U$ and thus $\hat{T}$ is well defined. The linearity of $\hat{T}$ is an easy consequence of the fact that $T$ is linear.
\begin{lemma}
 \label{lem:quotientspace}
 Let $\mathcal{W} =\{v \in W : Tv \in W\} \mbox{ and } \hat{\mathcal{W}}= \{\alpha \in W/U : \hat{T}(\alpha) \in~W/U\}$. Then $\hat{\mathcal{W}} = \mathcal{W}/ U.$  
\end{lemma}

\begin{proof}
 Note that $U\subseteq \mathcal{W}$. We have
  \begin{align*}
    v+U\in \hat{\mathcal{W}}&\iff v+U\in W/U \mbox{ and } Tv+U \in W/U\\
                             &\iff v\in W \mbox{ and } Tv\in W\\
    &\iff v\in \mathcal{W}. \qedhere
  \end{align*}
\end{proof}
\begin{lemma}
  Let $W$ be a proper subspace of an $n$-dimensional vector space $V$ over $\Fq$ and let $T\in L(W,V)$. Let $U$ denote the maximal $T$-invariant subspace and suppose dim $U= d$. Suppose $T\in \mathcal{C}(\lambda,\I)$ for some integer partition $\lambda \vdash n-d$. Then the linear transformation $\hat{T}: W/U \rightarrow V/U $ defined by $\hat{T} ( v+ U) = Tv + U$ is simple and $\hat{T}\in \mathcal{C}(\lambda,\emptyset)$.
\end{lemma}

\begin{proof}
  To show that $\hat{T}$ is simple, it suffices to show that the maximal invariant subspace of $\hat{T}$ is the zero subspace. Let $\{W_i\}_{i=0}^{\ell}$ be the chain of subspaces associated with $T$ 
  with $W_\ell =U$. Similarly, there is a chain of subspaces $\{\hat{W}_i\}_{i=0}^{\ell'}$ associated with $\hat{T}$. 
  It follows by Lemma \ref{lem:quotientspace} that $ \hat{W}_2=W_2/ U$. By applying the lemma again to the restriction of $\hat{T}$ to $ W_2/ U,$ we obtain $\hat{W}_3=W_3/ U$. By repeated application of the lemma it is clear that $\hat{W}_i = W_i/ U$ for $0 \leq i \leq \ell$. This implies that  $\ell'= \ell$ and that the maximal invariant subspace  $\hat{W}_\ell$ of $\hat{T}$  is the zero subspace. Thus $\hat{T} $ is simple. Since
$$ \dim W_{j-1}/ U - \dim W_{j}/ U = \dim W_{j-1}- \dim W_{j} = \lambda_j $$
for $1 \leq j \leq \ell$, the sequence of defect dimensions of $\hat{T}$ is $\lambda$.
\end{proof}
\begin{definition}
  For $T\in \mathcal{L}(V)$, the map $\hat{T}$ defined above is called the \emph{simple part} of $T$.  
\end{definition}

\begin{definition}
  For $T\in \mathcal{L}(V)$, the \emph{operator part} of $T$ denotes the linear operator obtained by restricting $T$ to its maximal invariant subspace.
\end{definition}


Given a subspace $W$ of $V$ and any $T\in L(W,V)$, associate with it a pair $(\overline{T},\hat{T})$ where $\overline{T}$ denotes the operator part of $T$ and $\hat{T}$ denotes the simple part of $T$. The following proposition asserts that the number of linear transformations having prescribed simple and operator parts is a power of $q$.

\begin{proposition}
\label{prop:pair}
Let $U\subseteq W$ be subspaces of an $n$-dimensional vector space $V$ over $\Fq$ and suppose that the dimensions of $U$ and $W$ are $d$ and $k$ respectively.
Let $T_o$ be a linear operator on $U$ with ordered set of invariant factors  $\I$ and let $T_s \in L(W/U, V/U) $ be a simple linear transformation with defect dimensions $\lambda \vdash n-d$. The number of linear transformations $T\in L(W,V)$ with operator part $T_o$ and simple part $T_s$ 
is given by $q^{d(k-d)} $.
\end{proposition}
\begin{proof}
  Let $\mathcal{B}=\{\alpha_1,\ldots,\alpha_d\}$ be an ordered basis for $U$. Extend $\mathcal{B}$ to a basis $\mathcal{B}'=\{\alpha_1,\ldots,\alpha_k\}$ for $W$. Let $T_o$ and $T_s$ be as in the statement of the theorem. If a linear transformation $T\in L(W,V)$ has operator part $T_o$, then $T$ is uniquely defined at each element of $\mathcal{B}$. It remains to define $T$ on each $\alpha_i$ for $d+1\leq i\leq k$. Suppose that $T_s(\alpha_i+U)=\beta_i+U$ for some $\beta_i\in V$ and $d+1\leq i\leq k$. Then $T\alpha_i+U=\beta_i+U$ for $d+1\leq i\leq k$. It therefore suffices to count maps $T$ satisfying
  $$
T(\alpha_i)=\beta_i+\gamma_i \mbox{ for some }\gamma_i \in U  \quad (d+1\leq i\leq k).
$$
The number of such maps is clearly $q^{d(k-d)}$.
\end{proof}

The function $\sigma(\lambda)$ defined in Corollary \ref{cor:cwv} counts the number of simple maps with defect dimensions $\lambda$ when $\lambda$ is a partition of a positive integer. As the simple part of any linear operator on $V$ is trivial, it is natural to extend the domain of definition of $\sigma(\lambda)$ to the empty partition by declaring $\sigma(\emptyset)=1$. 

\begin{theorem}
  \label{th:givenU}
Let $U\subseteq W$ be subspaces of an $n$-dimensional vector space $V$ over $\Fq$ and suppose $\dim U=d$ and $\dim W=k$. Let $\lambda \vdash n-d$ with $\lambda_1=n-k$ and $\I$ be an ordered set of invariant factors of degree $d$. The number of maps in $\CC_{W,V}(\lambda,\I)$ with maximal invariant subspace $U$ equals
 \begin{equation}
 \label{eqCC'}
  q^{d(k-d)}   \left|\CC(\I)\right|   \sigma(\lambda). 
 \end{equation}
\end{theorem}

\begin{proof}
There are precisely $\left|\CC(\I)\right|$ possibilities for the operator part of $T$. Setting $W'=W/U$ and $V'=V/U$, the simple part of $T$ can be chosen in $\left| \CC_{W',V'}(\lambda, \emptyset)\right|=\sigma(\lambda)$ ways. The result now follows from Proposition \ref{prop:pair}. 
\end{proof}

\begin{corollary}
  \label{cor:sizeofcwv}
Let $W$ be a $k$-dimensional subspace of an $n$-dimensional vector space $V$ over $\Fq$. Let $\I$ be an ordered set of invariant factors with $\deg \I=d\leq \dim W$ and let $\lambda\vdash n-d$ with $\lambda_1=n-k$. Then 
 \begin{equation}
 \label{eqC'}
        \left|\CC_{W,V}(\lambda,\I )\right|=     q^{d(k-d)}{k \brack d}_q  \left|\CC(\I)\right| \sigma(\lambda).
\end{equation}
  
\end{corollary}

\begin{proof}
The corollary follows from Theorem \ref{th:givenU} as there are ${k \brack d}_q$ possibilities for the maximal invariant subspace.   
\end{proof}
In the case $W=V$, the above expression for $\left|\CC_{W,V}(\lambda,\I )\right|$ reduces to $|\mathcal{C}(\I)|$, the number of square matrices whose invariant factors are given by $\I$. The next corollary determines the size of the similarity classes in $\mathcal{L}(V)$. 
\begin{corollary}
  \label{cor:simsize}
Let $V$ be a vector space over $\Fq$ of dimension $n$. If $\deg \I=d$ and $\lambda\vdash n-d$, then 
$$
 \left| \CC(\lambda,\I )\right| = q^{d(k-d)} {n \brack k}_q {k \brack d}_q  \left|\CC(\I)\right| \sigma(\lambda),
$$
where $k=n-\lambda_1$. 
\end{corollary}

\begin{proof}
Any map in $\mathcal{C}(\lambda,\I)$ has domain of dimension $k$. The result follows from Corollary \ref{cor:sizeofcwv} by summing $\left|\mathcal{C}_{W,V}(\lambda,\I)\right|$ over all $k$-dimensional subspaces of $V$. 
\end{proof}

The following result was proved in \cite[Thm. 3.8]{RAM2017146}. 
\begin{corollary}
Let $W$ be a fixed $k$-dimensional subspace of an $n$-dimensional vector space $V$ over $\Fq$. The number of linear transformations $T\in L(W,V)$ for which the operator part of $T$ has invariant factors $\I$ with $\deg \I=d$ equals
  $$
 {k \brack d}_q \left|\CC(\I)\right| \prod_{i=d+1}^k(q^n-q^i).
  $$
\end{corollary}

\begin{proof}
  By Corollary \ref{cor:sizeofcwv} the desired number of linear transformations equals
  \begin{align*}
    \sum_{\substack{\lambda \vdash n-d\\\lambda_1=n-k}}\left|\mathcal{C}_{W,V}(\lambda,\I)\right|&= q^{d(k-d)}{k \brack d}_q  \left|\CC(\I)\right| \sum_{\substack{\lambda \vdash n-d\\\lambda_1=n-k}}\sigma(\lambda)\\
                                                                                                 &= q^{d(k-d)} {k \brack d}_q \left|\mathcal{C}(\I)\right| \prod_{j=1}^{k-d}(q^{n-d}-q^j)\\
    &= {k \brack d}_q \left|\CC(\I)\right| \prod_{i=d+1}^k(q^n-q^i).
  \end{align*}
The second equality above is a consequence of Corollary \ref{cor:numsimp}.
\end{proof} 


\end{document}